\documentclass[11pt]{amsart}
\usepackage{color}
\usepackage[numbers,sort&compress]{natbib}
\usepackage{amsthm}
\usepackage{mathrsfs}
\usepackage{graphicx}
\usepackage{latexsym,bm,amsmath,amssymb}
\usepackage{epsfig}
\usepackage{calligra}
\usepackage[T1]{fontenc}
\allowdisplaybreaks[4]
\title[planar system defined by the sum of two quasi-homogeneous vecter fields]{Limit cycles of planar system defined by the sum of two quasi-homogeneous vecter fields}
\author[J.Huang and H.Liang]{Jianfeng Huang$^1$ and Haihua Liang$^2$}
\address{$^1$  Department of Mathematics,\ Jinan University,\ Guangzhou\ 510632,\ P.R.\ China}
\email{thuangjf@jnu.edu.cn}
\address{$^2$
School of Mathematics and Systems Science,\ Guangdong Polytechnic
Normal University,\ Guangzhou\ 510665,\ P.R.\ China}
\email{lianghhgdin@126.com}
\thanks{The first author is supported by the NSF of China (No.11401255) and the China Scholarship Council (No.201606785007)
and the Fundamental Research Funds for the Central Universities (No.21614325).
The second author is supported by the NSF of China (No.11771101) and the NSF of Guangdong Province (No.2015A030313669) and the Excellent Young Teachers Training Program for colleges and universities of Guangdong Province, China (No.Yq2013107).
}
\newtheorem{thm}{Theorem}[section]
\newtheorem{prop}[thm]{Proposition}
\newtheorem{lem}[thm]{Lemma}
\newtheorem{cor}[thm]{Corollary}

\theoremstyle{remark}

\theoremstyle{definition}

\newtheorem{rem}[thm]{Remark}

\begin{document}

\maketitle

\begin{abstract}
In this paper we consider the limit cycles of the  planar system
\begin{align*}
\frac{d}{dt}(x,y)=\bm X_n+\bm X_m,
\end{align*}
where $\bm X_n$ and $\bm X_m$ are quasi-homogeneous vector fields of degree $n$ and $m$ respectively. We prove that under a new hypothesis, the maximal number of limit cycles of the system is $1$.
We also show   that our result can be applied to some systems when the previous results are invalid.
The proof
is based on the investigations for the Abel equation and the generalized-polar equation associated with the system, respectively.
Usually these  two kinds of equations need to be dealt with separately, and for both equations, an efficient approach to estimate the number of periodic solutions is constructing suitable auxiliary functions.
  In the present paper we develop a formula on the divergence, which allows us to construct  an auxiliary function of one equation  with  the auxiliary function of the other equation, and
  vice versa.

\end{abstract}

\tableofcontents \setcounter{tocdepth}{1}

\section{Introduction and statements of main results}
Let $p,q\in\mathbb{Z}^{+}$ and $s\in\mathbb Z^+\cup\{0\}$. We say that a planar vector field $\bm{X}=(P,Q):\mathbb R^2\rightarrow\mathbb R^2$ is quasi-homogeneous with weight $(p,q)$ and degree $s$, if
\begin{align*}
P(\lambda^{p}x,\lambda^{q}y)=\lambda^{p+s-1}\cdot P(x,y),\ \
Q(\lambda^{p}x,\lambda^{q}y)=\lambda^{q+s-1}\cdot Q(x,y),\indent \lambda\in\mathbb{R}.
\end{align*}

There are plenty of works on the quasi-homogeneous vector field. A main purpose of them is to study the degenerated singularity and analyze the topology of a system. See \cite{taubes40,taubes41,taubes42,taubes43,taubes44,taubes45,taubes46,taubes47,taubes48,taubes49,taubes50} and the references therein. Observe that an arbitrary planar polynomial vector field with a singularity at origin can be written into the form $\bm X=\sum^n_{i=1}\bm X_i$, where $\bm X_i$ is quasi-homogeneous with weight $(p,q)$ and degree $i$. In particular, if $(p,q)=(1,1)$, then it is a well-known homogeneous decomposition.

One of the significant problems in the qualitative theory of real planar differential
systems is to control the number of limit cycles for a given class of polynomial systems, which is originated from the second part of Hilbert's 16th problem.
In this paper we consider the limit cycles of a planar system
\begin{align}\label{eq2}
\frac{d}{dt}(x,y)=\bm X(x,y)=\bm X_n(x,y)+\bm X_m(x,y),
\end{align}
where $m>n$, and $\bm X_i=(P_i,Q_i)$ is quasi-homogeneous vector field with weight $(p,q)$ and degree $i=n,m$.

As we know, system \eqref{eq2} in various types have been extensively studied and gained wide attention in decades \cite{CGGIJBC,taubes29,taubes30,taubes29,taubes2,taubes32,taubes42,taubes43,taubes51,taubes45,taubes32,taubes36,taubes31,taubes35,taubes33,taubes34}.
One of
the particularities of this system is that each limit cycle surrounding the origin can be expressed
in generalized polar coordinates as $r=r(\theta)$, with $r(\theta)$ being a smooth periodic function, see for instance \cite{taubes30,taubes29,taubes2,taubes32}, etc. This particularity provides us an opportunity to consider the Hilbert's 16th
problem in a natural and simple way.

So far, amount of work has been carried out for the bifurcation of \eqref{eq2} with small perturbations.  Li et al.  \cite{taubes42} provided an
upper bound for the number of limit cycles bifurcating from the period annulus of any homogeneous and quasi-homogeneous
centers. Following this line Gavrilov et al.  \cite{taubes43} and Gin\'{e} et al. \cite{taubes51} gave some significant improvements of the result.
In 1997, Cima et al.  \cite{taubes45} investigated the limit cycles for \eqref{eq2} in some special type. They proved that there exists such a system with at least $(n + m)/2$ limit cycles. Recently, Gasull et al. \cite{taubes32} improve the previous works. Their results include cases where the origin is a focus, a node, a saddle or a nilpotent singularity.
For more works, see for instance \cite{taubes36,taubes31,taubes35,taubes33,taubes34} and the references
therein.

In contrast, only a few results for \eqref{eq2} in non-bifurcation case are obtained. Here we summarize the representative ones as below: Let
\begin{align}\label{eq3}
\begin{split}
&a_i(\theta)=(m-n)\bigg(\cos\theta\cdot P_i(\cos\theta,\sin\theta)+\sin\theta\cdot Q_i(\cos\theta,\sin\theta)\bigg),\\
&b_i(\theta)=p\cos\theta\cdot Q_i(\cos\theta,\sin\theta)-q\sin\theta\cdot P_i(\cos\theta,\sin\theta),\\
&i=n,m.
\end{split}
\end{align}
\begin{itemize}
\item[(I)] If $a_mb_n-a_nb_m\not\equiv0$ does not change sign, then \eqref{eq2} has at most $1$ limit cycle. Moreover, this limit cycle surrounds the origin if it does exist (see Coll, Gasull and Prohens \cite{taubes29}).
\item[(II)] If $b_m\big(a_mb_n-a_nb_m\big)\not\equiv0$ does not change sign, then \eqref{eq2} has at most $2$ limit cycles surrounding the origin (see Carbonell and Llibre \cite{taubes30} for homogeneous case, and Coll, Gasull and Prohens \cite{taubes29} for general case).
\item[(III)] If $(p,q)=(1,1)$, $\bm X_n=(ax-y,x+ay)$, and $a_mb_n-2a_nb_m-\dot{b_m}\not\equiv0$ does not change sign, then \eqref{eq2} has at most $2$ limit cycles  surrounding the origin (see Gasull and Llibre \cite{taubes2}).
\item[(IV)] If $(p,q)=(1,1)$, $\bm X_n=(ax-y,x+ay)$, and either $a_mb_n-2a_nb_m-\dot{b_m}\equiv0$, or $b_m\big(a_mb_n-a_nb_m\big)\equiv0$, then \eqref{eq2} has at most $1$ limit cycle surrounding the origin (see Gasull and Llibre \cite{taubes2}).
\end{itemize}

This paper will concentrate on  system \eqref{eq2} with general weight $(p,q)$. As far as we know, until now, there is no other article
to study the upper bound for the number of limit cycles of this kind of system, except  \cite{taubes29}. In the excellent paper  \cite{taubes29},
the authors not only provide some sufficient conditions  such that  system  \eqref{eq2} can has at most $1$, or $2$, or $3$ limit cycles, but also   reveal many basic qualitative properties of the system.
These properties are of great help to understand the dynamical behaviors of  \eqref{eq2}.
However,
 during the applications we found that there exist many systems of the form \eqref{eq2} which do not fulfill the hypotheses proposed in \cite{taubes29}, but do have at most one   limit cycle.
 We will show some of  such  systems  in section 3.
The main goal of the present paper is providing a new criterion to estimate the number of limit cycles of system \eqref{eq2}, see the following theorem.


\begin{thm}\label{thm1}
Suppose that
\begin{align*}
\varPhi(\theta)\triangleq\frac{a_n(\theta)}{b_m(\theta)}-\dot{\left(\frac{b_n}{b_m}\right)}(\theta)\neq0.
\end{align*}
Then system \eqref{eq2} has at most $1$ limit cycle, counted with multiplicity. This upper bound is sharp. Furthermore, if the limit cycle
exists, then it surrounds the origin, and is stable (resp. unstable) when $b_m\varPhi>0$ (resp. $<0$).
\end{thm}


Recently, the authors in \cite{taubes37} provide a new uniqueness criterion for system \eqref{eq2} when $(p,q)=(1,1)$ and $\bm X_n=(ax-y,x+ay)$.
They prove that the maximal number of limit cycles surrounding the origin of the system is $1$, if $(m-1)ab_m+\dot{b_m}\neq0$. This result can be applied to some systems which violate the conditions stated in (I)-(IV).
Taking advantage of Theorem \ref{thm1}, we can easily improve it and obtain the following  corollary.



\begin{cor}\label{cor1}
Consider system
\begin{align}\label{eq9}
\begin{split}
&\frac{dx}{dt}=ax-y+P_m(x,y),\\
&\frac{dy}{dt}=x+ay+Q_m(x,y),
\end{split}
\end{align}
where $P_m, Q_m$ are homogeneous polynomials of degree $m\geq2$. Let
$$\psi(\theta)=\cos\theta\cdot Q_m(\cos\theta,\sin\theta)-\sin\theta\cdot P_m(\cos\theta,\sin\theta).$$
If $(m-1)a\psi(\theta)+\dot \psi(\theta)\neq0$, then the system has at most $1$ limit cycle (counted with multiplicity). This upper is sharp. Furthermore, if the limit cycle exists, then it surrounds the origin, and is stable (resp. unstable) when $(m-1)a+\dot{\psi}/\psi>0$ (resp. $<0$).
\end{cor}

A small imperfection of Theorem \ref{thm1} is that only a concrete system with one limit cycle is shown during the proof of the theorem. To overcome this shortage and ensure the existence of the limit cycle, we give the following criterion with a definite condition.
\begin{prop}\label{prop2}
System \eqref{eq2} has at least $1$ limit cycle surrounding the origin if
\begin{align*}
b_n(\theta)b_m(\theta)>0,\indent \int^{2\pi}_0\frac{a_n(\theta)}{b_n(\theta)}d\theta\cdot\int^{2\pi}_0\frac{a_m(\theta)}{b_m(\theta)}d\theta<0.
\end{align*}
\end{prop}

An example with exactly $1$ limit cycle will be provided to illustrate the application of Theorem \ref{thm1} and Proposition \ref{prop2} while all the previous  criterions (I)-(IV) are invalid. See section 3 for details.

\vspace{0.3cm}
There are several powerful tools to study system \eqref{eq2} in the papers mentioned above.
One of them is Abel equation. In particular, the classical investigations for non-bifurcation cases are mainly or partially based on the Abel equation with coefficients of definite signs.
We would like to point out that,
Theorem \ref{thm1} is partially obtained by a result on the Abel equation with coefficients changing signs, which is firstly given in \cite{taubes37}.

On the other hand, the necessary prerequisite under which we can change  system \eqref{eq2} to the Abel equation,  is $b_n\neq 0$, see Lemma 13 in \cite{taubes29} for instance.
 When $b_n$ vanishes at some points,  we
have to investigate directly  the generalized-polar equation associated with \eqref{eq2}, which seems more difficult.
Usually, in order to control the maximum number of isolated periodic solutions
for such equations, an efficient way is  to find some suitable auxiliary functions (see for example, Theorem 5 in \cite{taubes1}, or Lemma 2.1 in \cite{taubes37}, or Proposition 2.2 in \cite{taubes18}).

So, the other purpose of this paper, is to  build   a bridge    between the auxiliary functions of the Abel equation
and the generalized-polar equation associated with \eqref{eq2}. More precisely, we will develop a formula on the divergence, such that if an auxiliary function $F$
of the Abel equation is known for the case $b_n(\theta)\neq 0$, then in the situation that $b_n(\theta)$ has zero points, we can easily obtain a corresponding formal auxiliary function of the generalized-polar equation, namely $\overline{F}$.
This work will be done in section 2 and the   formula on the divergence will be shown in  Lemma \ref{lem2}.

The organization of this paper is as follows: In section 2 we
provide several preliminary results. The proofs of Theorem \ref{thm1}, Corollary \ref{cor1} and Proposition \ref{prop2} will be given in section 3.

\section{Preliminaries}\label{an
apprpriate label}
For the sake of proving our main result, we need several preparations. We shall state them one by one in the following.

\subsection{A basic result for 1-dimensional non-autonomous differential equation}\label{an
apprpriate label}
Consider an equation
\begin{align}\label{eq10}
\frac{dx}{dt}=L(t,x),
\end{align}
where $L\in\mathrm C^{\infty}([0,\kappa]\times\mathbb R)$ and $\kappa$ is a positive constant. Denote by $x(t,x_0)$ the solution of \eqref{eq10} with $x(0,x_0)=x_0$.

$x(t,x_0)$ is called a {\em periodic solution} of \eqref{eq10}, if it is defined in $[0,\kappa]$ with $x(\kappa,x_0)=x_0$. Moreover,
an orbit $x=x(t,x_0)$ in the strip $[0,\kappa]\times\mathbb R$ is called a {\em periodic orbit} (resp. {\em limit cycle}),
if $x(t,x_0)$ is a periodic solution (resp. isolated periodic solution) of the equation.

We call $H(x_0)=x(\kappa,x_0)$ the return map of \eqref{eq10}. It is well-known that
\begin{align}\label{eq40}
\begin{split}
&\dot{H}(x_0)=\exp\int^{\kappa}_0\frac{\partial L}{\partial x}\big(t,x(t,x_0)\big)dt,
\end{split}
\end{align}
where $\dot{}$ represent the first-order derivative (see Lloyd \cite{taubes4} for instance).
When $x(t,x_0)$ is periodic and $\dot H(x_0)\neq1$, $x=x(t,x_0)$ is called a hyperbolic limit cycle.

\begin{lem}\label{lem3}
Let $U$ be a planar region represented by $U=\big\{(t,x)\big|t\in[0,\kappa], x\in\big(c_1(t),c_2(t)\big)\big\}$, where $c_i\in\mathrm C^{0}([0,\kappa])\bigcup\{+\infty,-\infty\}$, $c_i(0)=c_i(\kappa)$ and $i=1,2$.
Let $F\in\mathrm C^1(U)$ with
$F(\kappa,x)=F(0,x)$.
Assume that
\begin{align*}
G(t,x)\triangleq
\frac{\partial L}{\partial x}(t,x)+\frac{\partial F}{\partial t}(t,x)+\frac{\partial F}{\partial x}(t,x)\cdot L(t,x),
\end{align*}
where $(t,x)\in U.$
\begin{itemize}
\item[(i)] If $x=x(t)$ is a periodic orbit of \eqref{eq10} in $U$, then
\begin{align}\label{eq39}
\begin{split}
\int^{\kappa}_{0}\frac{\partial L}{\partial x}\big(t,x(t)\big)dt=\int^{\kappa}_{0}G\big(t,x(t)\big)dt.
\end{split}
\end{align}
\item[(ii)] If $G\big|_{U}\geq0$ (resp. $\leq0$) and there exists a non-empty open set $E\subseteq[0,1]$ such that
$G\big|_{U\cap(E\times\mathbb R)}\neq0$,
then \eqref{eq10} has at most $1$ periodic orbit in $U$, which is hyperbolic unstable (resp. stable).
\end{itemize}
\end{lem}

The proof is easy and follows in a way similar to the proof of Lemma 2.1 in \cite{taubes37}.

\subsection{An estimate for the number of limit cycles of Abel equation}\label{an
apprpriate label}
Equation \eqref{eq10} is called Abel equation, if
\begin{align}\label{eq14}
L(t,x)=A(t)x^3+B(t)x^2+C(t)x\in\mathrm C^{\infty}([0,\kappa]\times\mathbb R).
\end{align}
Abel equation is one of the powerful tools to study system \eqref{eq2}. In what follows we state a second result in \cite{taubes37}, which is  essentially obtained by Lemma \ref{lem3}.

Let $W(\cdot,\cdot)$ represents the Wronskian determinant for two functions.
\begin{thm}\label{lem1}
Suppose that $L$ is defined as in \eqref{eq14} with $\kappa=1$. Suppose there exist two smooth functions $\lambda_1(t)>\lambda_2(t)$ such that $\lambda_i(t)\neq0$, $\lambda_i(1)=\lambda_i(0)$, $L(t,\lambda_i)-\dot{\lambda_i}$ does not change signs for $i=1,2$, and
\begin{align*}
4\lambda_1\lambda_2\left(L(t,\lambda_1)-\dot{\lambda_1}\right)\left(L(t,\lambda_2)-\dot{\lambda_2}\right)
+W^2(\lambda_1,\lambda_2)\leq0.
\end{align*}
Then \eqref{eq10} has at most $2$ non-zero limit cycles, counted with multiplicities.
\end{thm}

We would like to sketch the proof of Theorem \ref{lem1}, which is helpful to obtain our main result.

Firstly, it is not hard to verify that the conclusion is true if $L(t,\lambda_1)-\dot{\lambda_1}=L(t,\lambda_2)-\dot{\lambda_2}=W(\lambda_1,\lambda_2)\equiv0$.

When $L(t,\lambda_1)-\dot{\lambda_1}$, $L(t,\lambda_2)-\dot{\lambda_2}$ and $W(\lambda_1,\lambda_2)$ are not all identically zero, the key point is to find a suitable function $F(t,x)$ and then apply Lemma \ref{lem3}. Define $f(t,x)=(x-\lambda_1)(x-\lambda_2)x$. Based on the Lagrange interpolation formula, the authors in \cite{taubes37} take
\begin{align}\label{eq11}
F(t,x)=-\ln\left|\frac{f(t,x)}{\lambda_1\lambda_2}\right|,\indent (t,x)\in\big([0,1]\times\mathbb R\big)\setminus\big\{(t,x)\big|x=0,\lambda_1,\lambda_2,\ t\in[0,1]\big\}.
\end{align}
Therefore a direct calculation shows that
\begin{align}\label{eq12}
\begin{split}
\frac{\partial L}{\partial x}&(t,x)+\frac{\partial F}{\partial t}(t,x)+\frac{\partial F}{\partial x}(t,x)\cdot L(t,x)
=f(t,x)\cdot\frac{I_L(t,x)}{\lambda_1-\lambda_2},
\end{split}
\end{align}
where
\begin{align*}
I_L(t,x)=\sum^{2}_{i=1}\frac{(-1)^i\left(L(t,\lambda_i)-\dot{\lambda_i}\right)}{\lambda_i}\cdot\frac{1}{(x-\lambda_i)^2}
-\frac{W(\lambda_1,\lambda_2)}{\lambda_1\lambda_2}\cdot\frac{1}{(x-\lambda_1)(x-\lambda_2)}.
\end{align*}
It is easy to prove by assumption that $I_L\big|_{U}\geq0(\leq0)$, where $U$ represents an arbitrary connected component of $\big\{(t,x)\big|x\neq0,\lambda_1,\lambda_2,\ t\in[0,1]\big\}$. Moreover, there exists a non-empty open set $E\subseteq[0,1]$ such that
$I_L\big|_{U\cap(E\times\mathbb R)}\neq0$. Note that $f|_U\neq0$. According to Lemma \ref{lem3}, $U$ contains at most $1$ limit cycles of \eqref{eq10}, counted with multiplicity.

As a result, the number of non-zero limit cycles of \eqref{eq10} is no more than $6$. By virtue of further analysis (including bifurcation method and comparison principle), this upper bound can be reduced to $2$ and it is sharp. We obtain Theorem \ref{lem1}. For more details see \cite{taubes37}.
\vskip 0.3cm

Now let us consider system \eqref{eq2}. Take $$(x,y)=\left(r^{p/(m-n)}\cos(\theta),r^{q/(m-n)}\sin(\theta)\right).$$ The system becomes
\begin{align}\label{eq4}
\begin{split}
&\frac{dr}{dt}=\frac{a_n(\theta)+a_m(\theta)r}{p\cos^2(\theta)+q\sin^2(\theta)}\cdot r^{1+(n-1)/(m-n)},\\
&\frac{d\theta}{dt}=\frac{b_n(\theta)+b_m(\theta)r}{p\cos^2(\theta)+q\sin^2(\theta)}\cdot r^{(n-1)/(m-n)}.
\end{split}
\end{align}
It is known that the limit cycles surrounding the origin of system \eqref{eq2} do not intersect the curve $b_n(\theta)+b_m(\theta)r=0$ (see \cite{taubes29}), i.e. they are located in region
\begin{align}\label{eq16}
V\triangleq\big([0,2\pi]\times\mathbb R^+\big)\backslash\big\{(\theta,r)\big|b_n(\theta)+b_m(\theta)r=0\big\}.
\end{align}
Therefore, these limit cycles can be investigated by  equation
\begin{align}\label{eq5}
\frac{dr}{d\theta}=R(\theta,r)=\frac{a_n(\theta)r+a_m(\theta)r^2}{b_n(\theta)+b_m(\theta)r},\indent \theta\in[0,2\pi],\ r\in\mathbb R^+.
\end{align}

Furthermore, when $b_nb_m\neq0$, using the transformation originated
from Cherkas \cite{taubes28} and Coll et al. \cite{taubes29},
\begin{align}\label{eq53}
\rho=\frac{b_m(\theta)r}{b_n(\theta)+b_m(\theta)r}, \ \ \theta=2\pi\tau,
\end{align}
equation \eqref{eq5} is reduced to an Abel equation
\begin{align}\label{eq6}
\frac{d\rho}{d\tau}=S(\tau,\rho)\triangleq\alpha_3\big(\theta(\tau)\big)\rho^3+\alpha_2\big(\theta(\tau)\big)\rho^2+\alpha_1\big(\theta(\tau)\big)\rho,
\end{align}
where
\begin{align}\label{eq7}
\begin{split}
&\alpha_3(\theta)=2\pi\frac{a_n(\theta)b_m(\theta)-a_m(\theta)b_n(\theta)}{b_n(\theta)b_m(\theta)},\\
&\alpha_2(\theta)=2\pi\frac{a_m(\theta)b_n(\theta)-2a_n(\theta)b_m(\theta)+W(b_m,b_n)}{b_n(\theta)b_m(\theta)},\\
&\alpha_1(\theta)=2\pi\frac{a_n(\theta)b_m(\theta)-W(b_m,b_n)}{b_n(\theta)b_m(\theta)}.
\end{split}
\end{align}

Take
\begin{align}\label{eq13}
\lambda_1(\tau)=1,\ \ \lambda_2(\tau)=\varepsilon,\indent \varepsilon\in(-1/2,1/2)\backslash\{0\}.
\end{align}
We obtain the next proposition by Theorem \ref{lem1}.
\begin{prop}\label{prop3}
Under assumption of Theorem \ref{thm1}, if $b_n\neq0$, then system \eqref{eq2} has at most $1$ limit cycle surrounding the origin, counted with multiplicity. This upper bound is sharp.
\end{prop}
\begin{proof}
Clearly, $b_nb_m\neq0$ from assumption. In view of the above discussion, we only need to consider the limit cycles of equation \eqref{eq6}.

Let $\lambda_1$, $\lambda_2$ and $\varepsilon$ be defined as in \eqref{eq13}. Following a straightforward calculation,
\begin{align}\label{eq8}
\begin{split}
&W(\lambda_1,\lambda_2)=0,\\
& S(\tau,\lambda_1)-\frac{d\lambda_1}{d\tau}=0,\\
& S(\tau,\lambda_2)-\frac{d\lambda_2}{d\tau}=2\pi\left(\alpha_3(\theta)\varepsilon^2+\alpha_2(\theta)\varepsilon+
\frac{b_m(\theta)}{b_n(\theta)}\cdot\varPhi(\theta)\right)\varepsilon.
\end{split}
\end{align}
In addition, by assumption, for $0<|\varepsilon|<<1$,
\begin{align*}
\text{sgn}\left( S(\tau,\lambda_2)-\frac{d\lambda_2}{d\tau}\right)
=\text{sgn}\left(\varepsilon\cdot\frac{b_m(\theta)}{b_n(\theta)}\cdot\varPhi(\theta)\bigg)
\right)\neq0.
\end{align*}
Consequently, taking $\varepsilon\neq0$ sufficiently small, we know by Theorem \ref{lem1} that \eqref{eq6} has at most $2$ non-zero limit cycles, counted with multiplicities. In particular, \eqref{eq8} tells us that $\rho=1$ is one of them.

We emphasize that in $(\theta,r)$ coordinates, the limit cycles surrounding the origin of system \eqref{eq2} are all located in $V$, where $V$ is defined as in \eqref{eq16}. And transformation \eqref{eq53} sends $V$ to the region $U_1$ (resp. $U_2\cup U_3$) when $b_nb_m>0$ (resp. $b_nb_m<0$), where
\begin{align*}
&U_1=\big\{(\tau,\rho)\big|0<\rho<1,\ \tau\in[0,1]\big\},\\
&U_2=\big\{(\tau,\rho)\big|\rho<0,\ \tau\in[0,1]\big\},\\
&U_3=\big\{(\tau,\rho)\big|\rho>1,\ \tau\in[0,1]\big\}.
\end{align*}
Together with the above estimate for \eqref{eq6},
the number of limit cycles  surrounding the origin of system \eqref{eq2} is no more than $1$, counted with multiplicity.

Finally, take $n=2l+1$ and $m=2k+1$, where $k,l\in\mathbb Z^+\cup\{0\}$ and $k>l$.
Consider system \eqref{eq2} with
\begin{align*}
&\bm X_n=\left((x-y)\big(x^2+y^2\big)^l,(x+y)\big(x^2+y^2\big)^l\right),\\
&\bm X_m=\left(-(x+y)\big(x^2+y^2\big)^{k},(x-y)\big(x^2+y^2\big)^{k}\right).
\end{align*}
Then $\bm X_n$ and $\bm X_m$ are quasi-homogeneous with weight $(1,1)$.  A straightforward calculation shows that
\begin{align*}
a_n=2(k-l),\ a_m=-2(k-l),\ b_n=1,\ b_m=1.
\end{align*}
Thus, $\varPhi=2(k-l)>0$ and $b_n\neq0$. On the other hand, one can check that $x^2+y^2=1$ is a limit cycle of the system. Hence the upper bound is sharp.
\end{proof}

\subsection{A formula on the divergence}\label{an
apprpriate label}
Suppose that the assumption $b_n\neq0$ in Proposition \ref{prop3} is changed to $b_n=0$ for some $\theta\in[0,2\pi]$. Then equation \eqref{eq6} is not well-defined for some $\tau\in[0,1]$. For this reason, the estimate in the proposition is unable to be obtained by Abel equation and Theorem \ref{lem1}. We need to apply Lemma \ref{lem3} directly to the original equation \eqref{eq5}. The difficulty is still to find a suitable auxiliary function for the equation.

Nevertheless, the argument for the case $b_n\neq0$ provides us a clue to get this auxiliary function. To this end we need the following result.

\begin{lem}\label{lem2}
Let $T:V_1\rightarrow V_2$ be a diffeomorphism and $F\in\mathrm C^1(V_2)$, where the regions $V_1, V_2\subset\mathbb R^2$. Assume that $\bm Q$ and $\bm P$ are two vector fields on $V_1$ and $V_2$, respectively, with $\bm P\circ T=DT\cdot\bm Q$.
Then

\begin{align*}
(\mathrm{div}\bm P+D_{\bm P} F)\circ T
=\mathrm{div}\bm Q+D_{\bm Q}\overline{F}.
\end{align*}
where $D_{\bm Q}\cdot$ represents the directional derivative along $\bm Q$ and $\overline{F}=\ln|DT|+ F\circ T$.
\end{lem}
\begin{proof}
Denote by $\bm Q=(Q_1,Q_2)$, $\bm P=(P_1,P_2)$ and $\bm y=(y_1,y_2)=T(\bm x)=T(x_1,x_2)$. We have $P(\bm y)=\big(DT\cdot\bm Q\big)(\bm x)$. Hence a straightforward calculation shows that
\begin{align}\label{eq1.1}
&P_{1y_1}y_{1x_1}+P_{1y_2}y_{2x_1}=y_{1x_1x_1}Q_1+y_{1x_1x_2}Q_2+y_{1x_1}Q_{1x_1}+y_{1x_2}Q_{2x_1},\\ \label{eq2.1}
&P_{2y_1}y_{1x_1}+P_{2y_2}y_{2x_1}=y_{2x_1x_1}Q_1+y_{2x_1x_2}Q_2+y_{2x_1}Q_{1x_1}+y_{2x_2}Q_{2x_1},\\ \label{eq3.1}
&P_{1y_1}y_{1x_2}+P_{1y_2}y_{2x_2}=y_{1x_1x_2}Q_1+y_{1x_2x_2}Q_2+y_{1x_1}Q_{1x_2}+y_{1x_2}Q_{2x_2},\\ \label{eq4.1}
&P_{2y_1}y_{1x_2}+P_{2y_2}y_{2x_2}=y_{2x_1x_2}Q_1+y_{2x_2x_2}Q_2+y_{2x_1}Q_{1x_2}+y_{2x_2}Q_{2x_2}.
\end{align}
So, $y_{2x_2}\times$\eqref{eq1.1}$-y_{1x_2}\times$\eqref{eq2.1}$-y_{2x_1}\times$\eqref{eq3.1}$+y_{1x_1}\times$\eqref{eq4.1} leads to
\begin{align*}
|DT|\left(P_{1y_1}+P_{2y_2}\right)
=|DT|_{x_1}Q_1+|DT|_{x_2}Q_2+|DT|\left(Q_{1x_1}+Q_{2x_2}\right),
\end{align*}
which implies \begin{align*}
(\mathrm{div}\bm P)\circ T
=\mathrm{div}\bm Q+D_{\bm Q}\ln|DT|.
\end{align*}Furthermore, observe that
\begin{align*}
&D_{\bm P} F\circ T
=(\nabla F\circ T)\cdot(\bm P\circ T)
=(\nabla F\circ T)\cdot DT\cdot\bm Q
=\nabla(F\circ T)\cdot\bm Q
=D_{\bm Q}(F\circ T).
\end{align*}
The conclusion  of the lemma follows.
\end{proof}

Now we choose $T$ to be the transformation \eqref{eq53}, i.e.
\begin{align*}
T(\theta,r)=(\tau,\rho)=\left(\frac{\theta}{2\pi},\frac{b_m(\theta)r}{b_n(\theta)+b_m(\theta)r}\right).
\end{align*}
Take $V_1=V$ and $V_2=T(V)$, where $V$ is defined as in \eqref{eq16}.
Take $\bm P=\big(1, S(\tau,\rho)\big)$ and $\bm Q=\big(1,R(\theta,r)\big)$, i.e. the vector fields induced by equations \eqref{eq6} and \eqref{eq5}, respectively.

We begin to show the idea of finding a suitable auxiliary function for equation \eqref{eq5}.

Recall that Abel equation \eqref{eq6} and Theorem \ref{lem1} are the main tools during the proof of Proposition \ref{prop3} (i.e. the case $b_n\neq0$). In view of the sketch for the proof of theorem \ref{lem1}, an essential reason that Proposition \ref{prop3} holds is
\begin{align}\label{eq18}
\big(\mathrm{div}\bm P+D_{\bm P}F\big)(\tau,\rho)
=\frac{\partial S}{\partial \rho}&(\tau,\rho)+\frac{\partial F}{\partial \tau}(\tau,\rho)+\frac{\partial F}{\partial \rho}(\tau,\rho)\cdot S(\tau,\rho)
\neq0,
\end{align}
where $F$ is defined as in \eqref{eq11} with $\lambda_1=1$, $\lambda_2=\varepsilon$ and small $\varepsilon$.
On the other hand, under assumption of the proposition, we can verify that $T$ is a diffeomorphism from $V_1$ to $V_2$, and $F\in\mathrm C^1(V_2)$ for either $\varepsilon>0$ or $\varepsilon<0$. If we take $\varepsilon$ with suitable sign and denote
\begin{align*}
\overline F=\ln|DT|+F\circ T
=\ln|b_n+b_m r|-\ln|\varepsilon b_n-(1-\varepsilon)b_mr|-\ln|r|-\ln2\pi,
\end{align*}
then $\overline F$ is well-defined in $V_1=V$, and \eqref{eq18} and Lemma \ref{lem2} imply
\begin{align}\label{eq17}
\frac{\partial R}{\partial r}&\big(\theta,r\big)
+\frac{\partial \overline F}{\partial\theta}\big(\theta,r\big)+\frac{\partial \overline F}{\partial r}\big(\theta,r\big)\cdot R\big(\theta,r\big)
=\big(\mathrm{div}\bm Q+D_{\bm Q}\overline F\big)(\theta,r)\neq0.
\end{align}
Thus, the number of limit cycles of system \eqref{eq2} can also be estimated by using Lemma \ref{lem1} for equation \eqref{eq5} and $\overline F$ (although of the upper bound is not sharp, see Remark \ref{rem1}).

When the assumption $b_n\neq0$ in Proposition \ref{prop3} is changed to $b_n=0$ at some points, \eqref{eq17} can not be known from \eqref{eq18} and Lemma \ref{lem2} because the vector field $\bm P$ is not well-defined. Nevertheless, we will check by direct calculation. Considering that $\varepsilon$ is small enough in the previous case, we take $\varepsilon\rightarrow0$ for convenience, and therefore
\begin{align}\label{eq15}
\overline F\rightarrow\mathcal F\triangleq -\ln|b_m|+\ln|b_n+b_m r|-2\ln|r|-\ln2\pi.
\end{align}
Note that $\mathcal F$ is always well-defined in $V$. In what follows we apply Lemma \ref{lem3} directly to equation \eqref{eq5} with function $\mathcal F$, and give the next proposition.

\begin{prop}\label{prop4}
Under assumption of Theorem \ref{thm1}, if $b_n$ has zero points, then system \eqref{eq2} has at most $1$ limit cycle surrounding the origin, counted with multiplicity. This upper bound is sharp.
\end{prop}
\begin{proof}
Due to the previous argument, we only need to consider the periodic orbits of equation \eqref{eq5} which are located in the region $V$.
Furthermore, observe that $b_m\neq0$. The curve $b_n(\theta)+b_m(\theta)r=0$ can be represented by $r=-b_n(\theta)/b_m(\theta)$. From assumption, this curve intersects the curve $r=0$. That is to say, the region which contains the periodic orbits can be reduced to a simply connected region $U=\big\{(\theta,r)\big|r>0, r>-b_n(\theta)/b_m(\theta)\big\}\subset V$.

Now let $\mathcal F$ be the function determined by \eqref{eq15}. Then for $(\theta,r)\in U$, a direct calculation shows that
\begin{align}\label{eq19}
\begin{split}
\frac{\partial R}{\partial r}&\big(\theta,r\big)
+\frac{\partial\mathcal F}{\partial\theta}\big(\theta,r\big)+\frac{\partial\mathcal F}{\partial r}\big(\theta,r\big)\cdot R\big(\theta,r\big)
=-\frac{b_m(\theta)\cdot\varPhi(\theta)}{b_n(\theta)+b_m(\theta)r}
\neq0.
\end{split}
\end{align}
Hence Lemma \ref{lem3} tells us that the number of limit cycles of \eqref{eq5} in $U$ is no more than $1$, counted with multiplicity. Consequently, system \eqref{eq2} has at most $1$ limit cycle surrounding the origin (counted with multiplicity).

In order to prove that the upper bound can be achieved, we consider system \eqref{eq2} with
\begin{align*}
&\bm X_n=\left(\big(x^3-x^2y+xy^2\big)\big(x^2+y^2\big)^l,\big(x^3+x^2y+y^3\big)\big(x^2+y^2\big)^l\right),\\
&\bm X_m=\left(-(x+y)\big(x^2+y^2\big)^{k+1},(x-y)\big(x^2+y^2\big)^{k+1}\right),\\
&n=2l+1,\ m=2k+1,\ k>l,\indent k,l\in\mathbb Z^+\cup\{0\}.
\end{align*}
Clearly, $\bm X_n$ and $\bm X_m$ are quasi-homogeneous with weight $(1,1)$. It is easy to get that
\begin{align*}
a_n=2(k-l),\ a_m=-2(k-l),\ b_n=\cos^2\theta,\ b_m=1.
\end{align*}
Therefore $\varPhi=2(k-l)+\sin(2\theta)>0$ and $b_n$ has zero points. On the other hand, one can check that $x^2+y^2=1$ is a limit cycle of the system. This means that the upper bound is reachable.
\end{proof}

\begin{rem}\label{rem1}
We would like to point out that, the method used in the proof of Proposition \ref{prop4}, is not powerful enough to prove Proposition \ref{prop3}. In fact, when $b_mb_n<0$, $V$ is a region with two connected components. Therefore we can only know that system \eqref{eq2} has at most $2$ limit cycles surrounding the origin. This is why we have to apply Abel equation and Theorem \ref{lem1}.
\end{rem}

\subsection{The distribution of limit cycles of system \eqref{eq2}}
At the end of this section, we study the existence for the limit cycles which do not surround the origin.
\begin{prop}\label{prop1}
Under assumption of Theorem \ref{thm1}, any arbitrary periodic orbit of system \eqref{eq2} must surround the origin.
\end{prop}

\begin{proof}
Assume to the contrary that system \eqref{eq2} has a periodic orbit $\gamma$ which does not surround the origin. If we change the system into system \eqref{eq4}, then $\gamma$ in the $(\theta,r)$ coordinates is still a simple closed curve with $r>0$.

However, let $\overline{\bm X}$ be the vector field induced by \eqref{eq4}, and let
\begin{align*}
g(\theta,r)=\frac{p\cos^2\theta+q\sin^2\theta}{b_m(\theta)r^{2+{(n-1)/(m-n)}}}.
\end{align*}
It follows from assumption that $g$ is a smooth function in region $\{(\theta,r)|r>0\}$. In addition, for $r>0$ we have
\begin{align*}
\text{div}\left(g\overline{\bm X}\right)
=\frac{\partial }{\partial r}\left(\frac{a_n}{b_mr}+\frac{a_m}{b_m}\right)
+\frac{\partial}{\partial\theta}\left(\frac{b_n}{b_mr^2}+\frac{1}{r}\right)
=-\frac{\varPhi}{r^2}
\neq0.
\end{align*}
This implies that $\gamma$ does not exists, which shows a contradiction. As a result, our conclusion is true.
\end{proof}

\section{Proof of the main results}\label{an
apprpriate label}
There are two  goals  in this section. The first  is the proof of  Theorem \ref{thm1} and the second is  giving  some examples to illustrate the application
of our results when the previous criteria presented in (I)-(IV) are invalid.

\begin{proof}[Proof of Theorem \ref{thm1}]
The maximal number and distribution of the limit cycles of system
\eqref{eq2} is obtained directly by Proposition \ref{prop3},
Proposition \ref{prop4} and Proposition \ref{prop1}.

For the stability of the limit cycle of the system we again consider equation \eqref{eq5}.
Note that under assumption, \eqref{eq19} actually always holds in $V$.
Without loss of generality, suppose that $b_m\varPhi>0$ (it is a similar argument for the case $b_m\varPhi<0$). Then by Lemma \ref{lem3}, the limit cycle of \eqref{eq5} is stable (resp. unstable) if it is located in $V^+=\big\{(\theta,r)\big|b_n(\theta)+b_m(\theta)r>0\big\}$ (resp. $V^-=\big\{(\theta,r)\big|b_n(\theta)+b_m(\theta)r<0\big\}$). Observe that $d\theta/dt>0$ (resp. $<0$) in $V^+$ (resp. $V^-$) from \eqref{eq4}.
The limit cycle in $V^+$ (resp. $V^-$), as solutions of equation \eqref{eq5} and system \eqref{eq4} (i.e. \eqref{eq2}), has the  same (resp. opposite) stabilities. Hence, the limit cycle of \eqref{eq2} is stable.
\end{proof}

{\bf{Example 1.}}
Consider planar system
\begin{align}\label{eq16.3}
\begin{split}
&\frac{dx}{dt}=4x^3+xy^4-\left(2x^2+y^4\right)\left(8x+y^2\right)y,\\
&\frac{dy}{dt}=3x^2y+y^5+\left(2x^2+y^4\right)\left(x-4y^2\right).
\end{split}
\end{align}
We know that \eqref{eq16.3} is of the form \eqref{eq2} with $n=5$, $m=6$, $(p,q)=(2,1)$ and
\begin{align*}
&\bm X_5=(P_5,Q_5)=\bigg(4x^3+xy^4,3x^2y+y^5\bigg),\\
&\bm X_6=(P_6,Q_6)=\bigg(-\left(2x^2+y^4\right)\left(8x+y^2\right)y,\left(2x^2+y^4\right)\left(x-4y^2\right)\bigg).
\end{align*}
Thus, it is easy to check by \eqref{eq3} that
\begin{align*}
&a_5(\theta)=2\cos^2\theta+\sin^4\theta+\left(1+\cos^2\theta\right)\cos^2\theta,\\
&a_6(\theta)=-\left(2\cos^2\theta+\sin^4\theta\right)
\left(8-4\sin^2\theta-\cos^3\theta\right)\sin\theta,\\
&b_5(\theta)=\left(2\cos^2\theta+\sin^4\theta\right)\cos\theta\sin\theta,\\
&b_6(\theta)=\left(2\cos^2\theta+\sin^4\theta\right)^2.
\end{align*}

It can verify that $a_6(\theta)b_5(\theta)-a_5(\theta)b_6(\theta)$ is negative at $\theta=5\pi/4$ and positive at $\theta=\pi/2$. Therefore the representative results (I) and (II) in our introduction (section 1) are invalid for \eqref{eq16.3}.  Now we use  Theorem \ref{thm1} to determine the upper bound of the number of limit cycles of \eqref{eq16.3}.
Observe that
\begin{align*}
\frac{a_5(\theta)}{b_6(\theta)}-\dot{\left(\frac{b_5}{b_6}\right)}(\theta)
&=\frac{1}{2\cos^2\theta+\sin^4\theta}
   +\frac{\left(1+\cos^2\theta\right)\cos^2\theta}{\left(2\cos^2\theta+\sin^4\theta\right)^2}\\
&\ \ \ -\frac{\cos2\theta}{2\cos^2\theta+\sin^4\theta}
-\frac{4\cos^4\theta\sin^2\theta}{\left(2\cos^2\theta+\sin^4\theta\right)^2}\\
&=\frac{2\sin^2\theta}{2\cos^2\theta+\sin^4\theta}
 +\frac{\left(\cos^2\theta+\cos^2 2\theta\right)\cos^2\theta}{\left(2\cos^2\theta+\sin^4\theta\right)^2}\\
&>0.
\end{align*}
According to the theorem, system \eqref{eq16.3} has at most $1$ limit cycle.


\begin{proof}[Proof of Proposition \ref{prop2}]
Similar to the proof of Proposition \ref{prop3}, we obtain that $\rho=0$ and $\rho=1$ are two periodic orbits of \eqref{eq6}.
And transformation  \eqref{eq53} sends $V$ to the region $U_1$.
Note that the first derivative for the return map of \eqref{eq6} on $0$ and $1$ are
\begin{align*}
\exp\int^1_0\frac{\partial S}{\partial\rho}\left(\tau,0\right) d\tau
&=\exp\int^{2\pi}_0\left(\frac{a_n(\theta)}{b_n(\theta)}-\frac{\dot{b_n}(\theta)}{b_n(\theta)}
+\frac{\dot{b_m}(\theta)}{b_m(\theta)}\right) d\theta\\
&=\exp\int^{2\pi}_0\frac{a_n(\theta)}{b_n(\theta)}d\theta,\\
\exp\int^1_0\frac{\partial S}{\partial\rho}\left(\tau,1\right) d\tau
&=\exp\int^1_0\bigg(3 S(\tau,1)
-\alpha_2\big(\theta(\tau)\big)-2\alpha_1\big(\theta(\tau)\big)\bigg)d\tau\\
&
=\exp\int^{2\pi}_0\left(-\frac{a_m(\theta)}{b_m(\theta)}+\frac{\dot{b_n}(\theta)}{b_n(\theta)}
-\frac{\dot{b_m}(\theta)}{b_m(\theta)}\right)d\theta\\
&
=\exp\int^{2\pi}_0\left(-\frac{a_m(\theta)}{b_m(\theta)}\right)d\theta,
\end{align*}
respectively. By assumption, $\rho=0$ and $\rho=1$ are hyperbolic limit cycles with the same stability. For this reason, there exists at least $1$ limit cycle of \eqref{eq6} contained in $\big\{(\tau,\rho)\big|0<\rho<1,\ \tau\in[0,1]\big\}$. Again from the proof of Proposition \ref{prop3}, system \eqref{eq2} has at least $1$ limit cycle surrounding the origin.
\end{proof}

Here we provide a second example to compare Theorem \ref{thm1} with the representative results (I)-(IV) shown in section 1.

{\bf{Example 2.}}
Consider planar system
\begin{align}\label{eq16.1}
\begin{split}
&\frac{dx}{dt}=x-y-x^3+5x^2y-xy^2-y^3,\\
&\frac{dy}{dt}=x+y+3x^3-x^2y+9xy^2-y^3.
\end{split}
\end{align}
Then it is of the form \eqref{eq2} with $n=1$, $m=3$, $(p,q)=(1,1)$ and
\begin{align*}
&\bm X_1=(P_1,Q_1)=\left(x-y,x+y\right),\\
&\bm X_3=(P_3,Q_3)=\left(-x^3+5x^2y-xy^2-y^3,3x^3-x^2y+9xy^2-y^3\right).
\end{align*}
In addition, it follows from a direct calculation that
\begin{align*}
a_1(\theta)=2,\ \ a_3(\theta)=-2+8\sin2\theta,\ \ b_1(\theta)=1,\ \
b_3(\theta)=2+\cos2\theta.
\end{align*}
As a result,
\begin{align*}
&\frac{a_1(\theta)}{b_3(\theta)}-\dot{\left(\frac{b_1}{b_3}\right)}(\theta)
=\frac{4+2\cos2\theta-2\sin2\theta}{(2+\cos2\theta)^2}>0.
\end{align*}
Applying Theorem \ref{thm1}, the number of limit cycles of system \eqref{eq16.1} is at most $1$. In fact, we can obtain that $b_1(\theta)b_3(\theta)=2+\cos(2\theta)>0$ and
\begin{align*}
\int^{2\pi}_0\frac{a_1(\theta)}{b_1(\theta)}d\theta\cdot\int^{2\pi}_0\frac{a_3(\theta)}{b_3(\theta)}d\theta
=4\int^{2\pi}_0\frac{-1+4\sin2\theta}{2+\cos2\theta}d\theta
=-\int^{2\pi}_0\frac{4}{2+\cos2\theta}d\theta <0.
\end{align*}
Hence due to Proposition \ref{prop2} and Theorem \ref{thm1}, the system has exactly $1$ limit cycle, which surrounds the origin and is stable.

In contrast, for system \eqref{eq16.1} it also follows from a direct calculation that
\begin{align*}
&a_3b_1-a_1b_3=-6+8\sin2\theta-2\cos2\theta,\\
&b_3\big(a_3b_1-a_1b_3\big)=\big(2+\cos2\theta\big)\big(-6+8\sin2\theta-2\cos2\theta\big),\\
&a_3b_1-2a_1b_3-\dot{b_3}=-10+10\sin2\theta-4\cos2\theta.
\end{align*}
Obviously, all of these three equalities have indefinite signs, which violate the conditions of the representative results (I)-(IV). That is to say, our Theorem \ref{thm1} is indeed an important  supplement  of the previous works.

At the end of this section, we prove the following corollary and show that Theorem \ref{thm1} actually generalizes the main result given in \cite{taubes37}.

\begin{proof}[Proof of Corollary \ref{cor1}]
Firstly, since $\psi$ is differentiable and periodic, there exist $\theta_{\max}, \theta_{\min}\in[0,2\pi]$ such that
\begin{align*}
\psi(\theta_{\max})=\max_{\theta\in[0,2\pi]}\{\psi(\theta)\},\ \
\psi(\theta_{\min})=\min_{\theta\in[0,2\pi]}\{\psi(\theta)\},\ \
\dot{\psi}(\theta_{\max})=\dot{\psi}(\theta_{\min})=0.
\end{align*}
From assumption,
\begin{align*}
&(n-1)a\psi(\theta_{\max})=(n-1)a\psi(\theta_{\max})+\dot{\psi}(\theta_{\max})\neq0,\\
&(n-1)a\psi(\theta_{\min})=(n-1)a\psi(\theta_{\min})+\dot{\psi}(\theta_{\min})\neq0,
\end{align*}
which implies that $\psi(\theta_{\max})\neq0$, $\psi(\theta_{\min})\neq0$ and $a\neq0$.
Thus, $\text{sgn}\big(\psi(\theta_{\max})\big)=\text{sgn}\big(\psi(\theta_{\min})\big)\neq0$ (otherwise there exists $\theta_0\in[0,2\pi]$ satisfying $\dot{\psi}(\theta_{0})+(n-1)a\psi(\theta_{0})=0$). We obtain $\psi(\theta)\neq0$.

Secondly, system \eqref{eq9} is of the form \eqref{eq2} with $n=p=q=1$ and $\bm X_1=(P_1,Q_1)=(ax-y,x+ay)$. Hence we get by \eqref{eq3} that
\begin{align*}
a_1(\theta)=(m-1)a,\indent b_1(\theta)=1,\indent b_m=\psi(\theta).
\end{align*}
This leads to
\begin{align*}
&\frac{a_1(\theta)}{b_m(\theta)}-\dot{\left(\frac{b_1}{b_m}\right)}(\theta)
=\frac{1}{\psi^2(\theta)}\left((m-1)a\psi(\theta)+\dot{\psi}(\theta)\right)\neq0,\\
&a_1(\theta)-b_m(\theta)\dot{\left(\frac{b_1}{b_m}\right)}(\theta)
=(m-1)a+\frac{\dot{\psi}(\theta)}{\psi(\theta)}.
\end{align*}
According to Theorem \ref{thm1}, our assertion is true.
\end{proof}

\end{document}